\documentclass{amsart}
\usepackage[utf8]{inputenc}
\usepackage{amsmath,color,url,amsthm,amssymb, graphicx}
\usepackage{appendix}
\usepackage{tikzit,todonotes}

\tikzstyle{dot}=[fill=black, draw=black, shape=circle]
\tikzstyle{smalldot}=[fill=black, draw=black, shape=circle, minimum size=1mm]

\tikzstyle{dashed}=[-, thick, dashed]
\tikzstyle{invisfillblue}=[-, draw=none, fill={rgb,255: red,78; green,158; blue,255}]
\tikzstyle{invisfillred1}=[-, draw=none, fill={rgb,255: red,255; green,47; blue,47}]
\tikzstyle{invisfillred2}=[-, draw=none, fill={rgb,255: red,255; green,104; blue,58}]
\tikzstyle{invisfillgreen}=[-, draw=none, fill={rgb,255: red,70; green,204; blue,86}]
\tikzstyle{invisfillyellow}=[-, draw=none, fill=yellow]
\tikzstyle{arrowline}=[<->, thick]

\newtheorem{theorem}{Theorem}
\newtheorem{lemma}[theorem]{Lemma}
\title{Graphs of low average degree without independent transversals}
\thanks{The first author was supported by the European Research Council Horizon 2020 project CRACKNP (grant agreement no. 853234). The second author was supported by project GA20-09525S of the Czech Science Foundation.}

\author{Carla Groenland}
\address{Utrecht University, The Netherlands}
\email{c.e.groenland@uu.nl}

\author{Tom\'{a}\v{s} Kaiser}
\address{University of West Bohemia, Pilsen, Czech Republic}
\email{kaisert@kma.zcu.cz}

\author{Oscar Treffers}
\address{Radboud University Nijmegen, The Netherlands}
\email{}

\author{Matthew Wales}
\address{DPMMS, University of Cambridge, UK}
\email{mw637@cam.ac.uk}

\date{\today}
\newcommand{\N}{\mathbb{N}}

\newtheorem{problem}[theorem]{Problem}
\newtheorem{remark}[theorem]{Remark}

\newcommand\Setx[1]{\{#1\}}
\newcommand\size[1]{|#1|}
\newcommand\PP {\mathcal P}

\newcommand\eps {\varepsilon}

\begin{document}

\maketitle

\begin{abstract}
    An independent transversal of a graph $G$ with a vertex partition $\mathcal P$ is an independent set of $G$ intersecting each block of $\mathcal P$ in a single vertex. Wanless and Wood proved that if each block of $\mathcal P$ has size at least $t$ and the average degree of vertices in each block is at most $t/4$, then an independent transversal of $\mathcal P$ exists. We present a construction showing that this result is optimal: for any $\varepsilon > 0$ and sufficiently large $t$, there is a forest with a vertex partition into parts of size at least $t$ such that the average degree of vertices in each block is at most $(\frac14+\eps)t$, and there is no independent transversal. This unexpectedly shows that methods related to entropy compression such as the Rosenfeld-Wanless-Wood scheme or the Local Cut Lemma are tight for this problem. Further constructions are given for variants of the problem, including the hypergraph version. 
\end{abstract}

\section{Introduction}
\label{sec:introduction}

Graphs in this paper contain no parallel edges or loops. Let $G$ be a
graph and let $\PP$ be a partition of its vertex set $V(G)$ (also
called a \emph{vertex partition} of $G$). The sets in $\PP$ are called
the \emph{blocks} of $\PP$. Following~\cite{F90}, we say that $\PP$ is
\emph{$t$-thick} if the size of each block of $\PP$ is at least
$t$. An \emph{independent transversal} of $\PP$ in $G$ is an
independent set in $G$ intersecting each block of $\PP$ in a single
vertex.

Conditions for the existence of independent transversals in $G$ in
relation to the maximum degree $\Delta(G)$ of $G$ date back to a 1972
conjecture of Erd\H{o}s~\cite{E72}: if $\PP$ is $t$-thick and
$\Delta(G)\leq t-1$, then there is an independent transversal of $\PP$
in $G$. The conjecture was disproved by Seymour
(cf.~\cite{BES75}). 

Bollob\'{a}s, Erd\H{o}s and Szemer\'{e}di~\cite{BES75} constructed graphs $G$ with a $t$-thick
partition $\PP$ with $n$ blocks (for any $n\geq 3$) such that $\Delta(G) = \frac t2 + \frac t{n-2}$ and $\PP$
has no independent transversal in $G$. They conjectured that for any
$\varepsilon > 0$ and large enough $\size\PP$, $\Delta(G) \leq (\frac12-\eps)t$
ensures the existence of an independent transversal of a $t$-thick
partition $\PP$.

Alon~\cite{A88} and Fellows~\cite{F90} independently proved that there
is a constant $\alpha > 0$ such that any $t$-thick vertex partition of a
graph $G$ with $\Delta(G) \leq \alpha t$ admits an independent
transversal (with $\alpha=\frac1{25}$ and $\alpha=\frac1{16}$, respectively). Both
proofs are based on the Lov\'{a}sz Local Lemma. An improvement to
$\alpha=\frac1{2e}$ using the same method was proved by Alon~\cite{A94}
(see also~\cite[Chapter 5.5]{AS15}).

Constructions in~\cite{J92,Y97} show that for any $d$ which is a power
of two, there is a graph $G$ with a $(2d-1)$-thick partition $\PP$
such that $\Delta(G) = d$ and $\PP$ has no independent
transversal. (This was later extended to every $d$ by Szab\'{o} and
Tardos~\cite{ST06}.) Consequently, the above constant $\alpha$ must be
at least $\frac12$.

Haxell~\cite{H01} finally proved the optimal sufficient condition for
the existence of an independent transversal in terms of maximum degree.
\begin{theorem}[\cite{H01}, Theorem 2]
\label{t:haxell}
  Let $\PP$ be a $t$-thick vertex partition of a graph $G$ with
  $\Delta(G) \leq \frac t2$. Then $\PP$ has an independent transversal in
  $G$.
\end{theorem}

There is a close connection between independent transversals in a
graph and topological properties of its `independence complex'. These
connections were first uncovered in~\cite{AH00} and further developed
in e.g.~\cite{AB06,ABZ07,M01,M03}. Another closely related notion is
that of strong colouring of graphs which was introduced
in~\cite{A88,F90} (see~\cite{H08} and the references therein).

Reed and Wood~\cite{RW12} observed that the Lov\'{a}sz Local Lemma
implies a version of Theorem~\ref{t:haxell} in terms of average degree
rather than maximum degree. Given a graph $G$ and a partition $\PP$ of
$V(G)$, let us say that the degree $d(B)$ of a block $B\in\PP$ is the
number of edges with (precisely) one endvertex in $B$. Furthermore,
the \emph{maximum block average degree} of $\PP$ of $V(G)$ is the
maximum, over all blocks $B$, of $d(B)/\size B$. Reed and Wood's
result then states that if $\PP$ is a $t$-thick partition of $V(G)$
with maximum block average degree at most $\frac t{2e}$, then $\PP$ has an
independent transversal in $G$.

In~\cite{DEKO21}, an equivalent problem related to the single-conflict
chromatic number is considered, and it is noted that a stronger
sufficient condition with $\frac t{2e}$ replaced by $\frac t4$ follows from the
Local Cut Lemma~\cite{B17}. 

Another approach which yields the same improvement is the one of
Wanless and Wood~\cite{WW20}, who apply an idea of
Rosenfeld~\cite{R20} to a number of problems related to hypergraph
colouring, obtaining simpler proofs and sometimes stronger results
compared to earlier applications of the Lov\'{a}sz Local Lemma. Their
result on independent transversals is stated in a more general form
for hypergraphs. Let us restrict our attention to the graph case
which reads as follows.

\begin{theorem}[\cite{WW20}, Theorem 9]\label{t:ww-graphs}
  If $G$ is a graph and $\PP$ a $t$-thick partition of $V(G)$ with
  maximum block average degree at most $\frac t4$, then there exist
  $(\frac t2)^{\size\PP}$ independent transversals of $\PP$ in $G$.
\end{theorem}

It is natural to ask if the factor $\frac 14$ can be improved, perhaps even
to $\frac 12$ so as to obtain a genuine strengthening of
Theorem~\ref{t:haxell} in terms of maximum block average degree. This
is mentioned as an open problem in~\cite{WW20} as well as
in~\cite{KK20}.

The main result of this paper answers this question, showing that the
bound in Theorem~\ref{t:ww-graphs} cannot be improved. This has the
unexpected consequence that methods related to entropy compression such as the Rosenfeld--Wanless--Wood scheme or the Local Cut Lemma perform optimally for the given problem. To the
best of our knowledge, this is the first non-trivial instance of a
problem where this is known to be the case.

\begin{theorem}\label{t:main}
  For all $\eps>0$ and sufficiently large $t$, there exists a forest $F$
  and a $t$-thick vertex partition with maximum block average degree at
  most $(\frac14+\eps)t$ and no independent transversal in $F$.
\end{theorem}

Thus, if we let $\eps$ decrease to $0$, there is a sharp transition: while maximum block average degree of $(\frac 14+\eps)t$ does not imply an independent transversal for arbitrarily small nonzero $\eps$,  exponentially many are guaranteed at $\eps=0$ by Theorem~\ref{t:ww-graphs}. 

In Section~\ref{sec:forests_simple} we prove Theorem \ref{t:main}.
In Section~\ref{sec:max-degree} we show that the average degree bound of Theorem~\ref{t:main} can be
combined with an upper bound on the maximum degree of the graph.
Section~\ref{sec:hypergraphs} discusses an extension of Theorem~\ref{t:main} to
hypergraphs, which again demonstrates the optimality of the general bound
obtained in~\cite{WW20}. We finish with various open problems in Section~\ref{sec:concl}.

\section{Low average degree forests}
\label{sec:forests_simple}
In this section, we prove Theorem \ref{t:main}.
Let $G$ be a graph with a vertex partition $\PP$.
We say a vertex of $G$ is \emph{forbidden} if it cannot appear in an independent transversal of $\PP$. If we can find a block of $\PP$ in which every vertex is forbidden, this means no independent transversal exists.
We say a subset $S$ of $V(G)$ is \emph{forced} if any independent transversal contains (at least one) vertex from $S$. As an example, if $B\in\PP$ and all vertices in $A\subseteq B$ are forbidden, then $B\setminus A$ is forced.

Any vertex $v$ which is joined entirely to a forced subset $S$ is forbidden: an independent transversal must contain a vertex from $S$, and $v$ is adjacent to every vertex in $S$, so a transversal containing $v$ cannot be independent.

A subset $B\in\PP$ is always forced. Suppose that a vertex $v\in B'\in \PP$ is joined entirely to some block $B$. Then $v$ is forbidden and therefore $B'\setminus \{v\}$ is forced and of strictly smaller size. The idea of our construction is to iterate this phenomenon. 

\subsection{The construction}
\label{subsec:forests_constr}
Let $0=n_1\leq \dots\leq n_k=t$ be a given sequence of integers for some $t>0$. 
We inductively define forests $F_1,F_2,\dots,F_k$ which consist of blocks of several `grades'; these blocks define the desired vertex partition. Each block will consist of some number of \textit{light} vertices of degree at most 1, and some number of \textit{heavy} vertices of high degree. The construction is illustrated in Figure \ref{fig:1}.
\begin{itemize}
    \item Let $F_1$ be an independent set of size $t$. The vertex partition consists of a single block of grade 1 (consisting entirely of light vertices).
    \item Let $F_2$ consist of $n_2$ copies of $F_1$ plus $t$ vertices in a single block of grade 2. The vertices in the block of grade 2 consist of $t-n_2$ light vertices and $n_2$ heavy vertices. Each heavy vertex is fully connected to a private copy of $F_1$.
\item Having defined $F_j$ for some $2\leq j<k$, we define $F_{j+1}$ as follows. 
    Let $F_{j+1}$ consist of $n_{j+1}$ copies of $F_j$ and a single block of grade $j+1$, inducing an independent set of size $t$. The vertices in the block of grade $j+1$ consist of $t-n_{j+1}$ light vertices and $n_{j+1}$ heavy vertices, which are each fully adjacent to the $t-n_j$ light vertices in the unique block of grade $j$ in a private copy of $F_j$. This defines a unique block of grade $j+1$ in $F_{j+1}$, and it inherits the partition into blocks from the copies of $F_j$ for the remaining vertices. 
\end{itemize}

Within each forest $F_i$, every heavy vertex is forbidden. This can be proven by induction on the grade $j\in \{2,\dotsc,i\}$; each heavy vertex of grade 2 is joined to an entire block of grade 1, so is forbidden. If the heavy vertices of grade $j$ are forbidden, this means the light vertices of a grade $j$ block form a forced set, and so each heavy vertex of grade $j+1$ is forbidden. In a block of grade $k$ within $F_k$ all vertices are heavy, and so no independent transversal exists.
\begin{figure}
    \centering
    \begingroup
        \tikzset{every picture/.style={scale=0.7}}
        \input{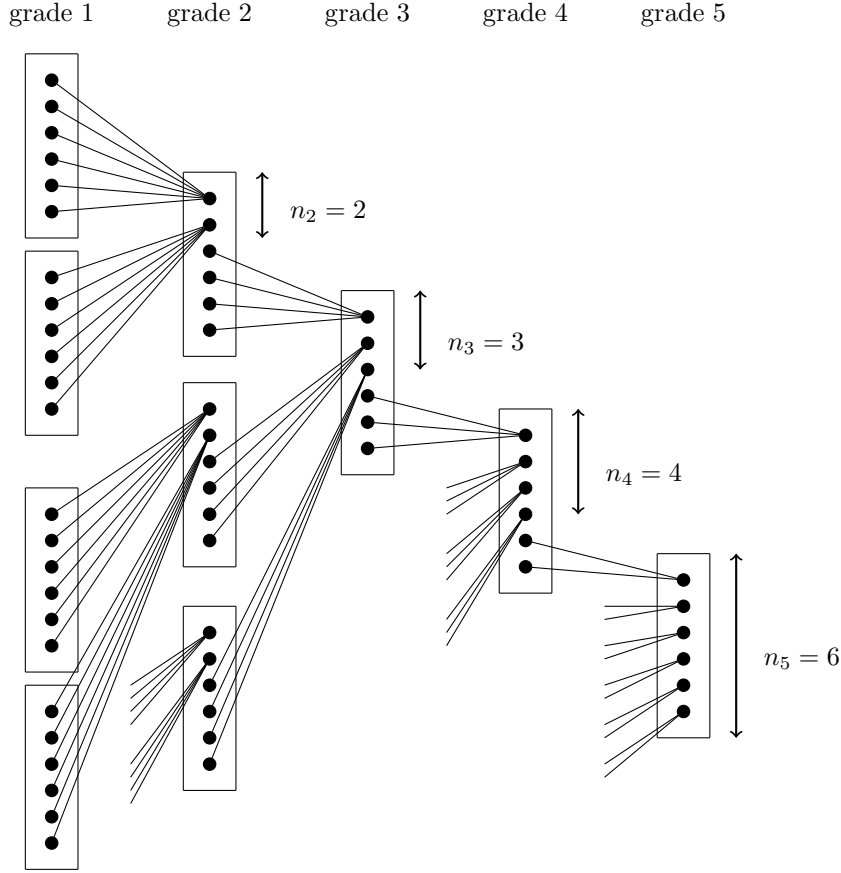}
    \endgroup
    \caption{The construction of Section \ref{sec:forests_simple} is illustrated. The square boxes denote entire blocks of that grade.}
    \label{fig:1}
\end{figure}
\subsection{Analysis}
To prove Theorem \ref{t:main}, it remains to show that there is a sequence of integers $(n_i)_{i=1}^k$ such that the construction of the previous subsection has low block average degree. This is shown in the following lemma.
\begin{lemma}
\label{lem:n}
For every $\varepsilon>0$, there exists $t_0\in \N$ such that for all $t\geq t_0$, there is a sequence of integers $0=n_1< \dots < n_k=t$ such that for all $j\in \{1,\dots,k-1\}$
\begin{equation}
    \label{eq:averagedegree}
\frac1t(n_{j+1}\cdot (t-n_{j})+(t-n_{j+1})\cdot 1)\leq \left(\frac{1}{4}+\varepsilon\right)t.
\end{equation}
\end{lemma}
\begin{proof}
Let $\varepsilon>0$ be given.
Let $\delta>0$ and $t_1\in \N$ be such that
\begin{equation}
    \label{eq:delta_to_eps}
\left(\frac{1}{4}+\delta\right) t+1\leq \left(\frac{1}{4}+ \varepsilon\right)t
\end{equation}
for all $t\geq t_1$. 
Let $t_0\geq t_1$ such that $\frac{\delta t}{2} > 1$ for all $t\geq t_0$. 

Let $t\geq t_0$ be given. We set $n_1 = 0$. For $j\geq 1$, while $n_{j}<\frac34 t$, we set $n_{j+1}$ to be the minimum of $\lfloor \frac{1/4 +\delta}{1-n_j/t}t\rfloor$ and $t$. If $n_j=t$ we set $k = j$ and abort the procedure.
It remains to show (\ref{eq:averagedegree}) holds and that the procedure terminates.

We first show (\ref{eq:averagedegree}). 
For all integers $j\in \{1,\dots,k-1\}$, we find using the definition of $n_j$ and (\ref{eq:delta_to_eps}) that
\begin{align*}
\frac1t(n_{j+1}\cdot (t-n_{j}) +(t-n_{j+1})\cdot 1) &\leq
\frac1t
\frac{(1/4+\delta)t}{1-n_j/t}(t-n_{j})
+1\\
&= \left(\frac{1}{4}+\delta\right)t+1 \\
& \leq \left(\frac{1}{4}+\varepsilon\right)t.
\end{align*}
We now show that $n_{j+1}>n_j$ if $n_j<t$, which implies that the procedure terminates.
We need to be slightly careful in our analysis due to the floor in the definition of $n_{j+1}$.
Since $\frac{\delta t}{2}>1$, the interval 
from $(\frac14+\frac{\delta}2)\frac{t}{1-x}$ to $(\frac14+\delta)\frac{t}{1-x}$ has width greater than $1$ for all $x<1$, and so for $n_j<t$ we find
\[
\frac{n_{j+1}}{t}\geq \frac{1/4+\delta/2}{1-n_{j}/t}.
\]
We use the fact that $\frac{1/4+0}{1-x} \geq x$ for $x<1$ (which can be shown by proving that $x(1-x)$ has a unique maximum of $\frac{1}{4}$ at $x = \frac{1}{2}$). For $n_j<t$, this fact implies that
\[
\frac{1/4+\delta/2}{1-n_{j}/t}\geq \frac{n_j}{t}+\frac{\delta/2}{1-n_{j}/t}>\frac{n_j}{t}.
\]
Combining the two displayed inequalities above shows that $n_{j+1}>n_j$ if $n_j<t$. This implies that there exists a $j$ for which $n_j\geq \frac34t$, as desired.
\end{proof}
\begin{remark}
\label{R:reviewercomment}
One could instead take $n_j = j$ for each of $j=1,\dotsc,t$. Then $n_j>n_{j-1}$ and inequality $(\ref{eq:averagedegree})$ holds by the following sequence of inequalities for $t>4/\epsilon$:
\begin{align*}
    \frac1t\left((j+1)(t-j) + (t-j-1)\right) \leq \frac1t \left(j(t-j) + 2t\right) \leq \frac14t + 2\leq \left(\frac14+\epsilon\right)t.
\end{align*}
This gives a simplified proof of Lemma~\ref{lem:n}, however we include the above argument since this generalises more easily to the maximum degree setting.
\end{remark}

\begin{proof}[Proof of Theorem \ref{t:main}]
    Let $\varepsilon>0$. Let $t_0\in \mathbb{N}$ be given from Lemma \ref{lem:n}. For every $t\geq t_0$, we obtain from the lemma a sequence of integers $0=n_1\leq \dots\leq n_k=t$ such that (\ref{eq:averagedegree}) holds for all $j\in \{1,\dots,k-1\}$. We use this sequence of integers in the construction of Section \ref{subsec:forests_constr} and claim that the graph $F_k$ with its vertex partition satisfies the requirements. We already remarked that it is a forest and that it does not admit an independent transversal, so it suffices to verify that each block has average degree at most $(\frac14+\varepsilon)t$. For $j\in \{1,\dots,k-1\}$, a block of grade $j+1$ consists of $n_{j+1}$ heavy vertices of degree $t-n_{j}$ and $t-n_{j+1}$ light vertices of degree $1$. Therefore, the average degree of these blocks is at most $(\frac14+\varepsilon)t$ by (\ref{eq:averagedegree}). A block of grade 1 consists of $t$ vertices of degree $1\leq (\frac14+\varepsilon)t$ (for $t\geq 4$), so indeed all blocks have average degree at most $(\frac14+\varepsilon)t$ for $t$ sufficiently large.
\end{proof}

The termination of the procedure in the proof of Lemma \ref{lem:n} can be predicted by studying the behaviour of the M\"{o}bius transformation \mbox{$f:z\mapsto \frac{\alpha}{1-z}$} for $\alpha\in [\frac14,1)$. For $\alpha\in (\frac14,1)$ the transformation is conjugate to a rotation $z\mapsto kz$ for some $k\in \mathbb{C}$ with $|k|=1$ and $k\neq 1$ (see \cite[Theorem 4.3.4]{Beardon} for more information). 
A M\"obius transformation gets determined by the image of any three of its points and maps generalised circles to generalised circles. Since $0,1,\infty$ get mapped to $\alpha,\infty,0$ respectively, the orbit under iteration of $f$ of any element on the real line must contain some element of $[1,\infty)$. 
The behaviour of the dynamical system determined by $f$ changes dramatically at $\alpha=\frac14$: here it has an attracting fixed point at $\frac12$ in this case, and this is the reason that for such $\alpha$ we cannot find a sequence $(n_{j})_{j=0}^k$ satisfying the assumptions of Lemma \ref{lem:n}. This aligns nicely with the fact that when the average degree is at most $\frac14t$, an independent transversal must always exist.
 
\section{Bounding the maximum degree}
\label{sec:max-degree}
In our low block average degree forest construction in Section \ref{subsec:forests_constr}, the grade 2 blocks contained vertices joined entirely to a block, so these vertices have degree $t$. Is this the only obstruction to finding an independent transversal, or can we find an example of a construction where the maximum degree of a vertex is strictly smaller? In this section, we resolve this question by presenting a slightly modified version of our earlier construction. We will only highlight the modifications.
    
    \begin{theorem}
    \label{thm:maxdegree}
        For all $\eps>0$ and $t$ sufficiently large, there is a graph with a $t$-thick vertex partition, maximum block average degree at most $(\frac14 + \eps)t$ and maximum degree at most $0.854t$.
    \end{theorem}
    
    \begin{proof}
        In the construction of Section \ref{subsec:forests_constr}, the vertices of maximum degree arise because the forced sets at grade 1 form the entire block. If we can replace this with a smaller forced set, while maintaining the same average degree condition, we can decrease the maximum degree.
        
        We define the new blocks of grade 1 as follows. Consider two vertex sets $V_a,V_b$ of size $t$, which form two separate parts in our vertex partition but together form a block of grade 1. We create a complete bipartite graph between $A\subseteq V_a$ and $B\subseteq V_b$. The number of edges between blocks $V_a$ and $V_b$ is $|A||B|$, and no independent transversal can contain both a vertex from $A$ and a vertex from $B$. Therefore, $V_a\cup V_b - (A\cup B)$ is a forced set of size $2t-|A|-|B|$.
        
        This defines a graph $F_1'$: a copy of $K_{|A|,|B|}$ (consisting of heavy vertices) together with $2t-|A|-|B|$ isolated (light) vertices. The associated block partition will be $A$, together with $t-|A|$ isolated vertices, and the remaining $t$ vertices. 
       In the previous construction, we will replace each copy of $F_1$ with a copy of $F_1'$. This is depicted in Figure \ref{fig:2}.
    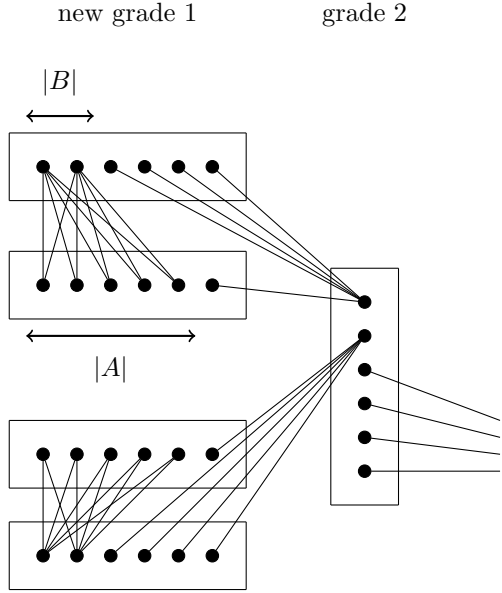
\begin{figure}
    \centering
    \begingroup
        \tikzset{every picture/.style={scale=0.9}}
        \begin{tikzpicture}
	\begin{pgfonlayer}{nodelayer}
		\node [style=dot, scale=0.5] (10) at (-7, 16) {};
		\node [style=dot, scale=0.5] (11) at (-7, 16.5) {};
		\node [style=dot, scale=0.5] (12) at (-7, 17) {};
		\node [style=dot, scale=0.5] (13) at (-7, 17.5) {};
		\node [style=dot, scale=0.5] (14) at (-7, 18) {};
		\node [style=dot, scale=0.5] (15) at (-7, 18.5) {};
		\node [style=none] (16) at (-7.5, 19) {};
		\node [style=none] (17) at (-6.5, 19) {};
		\node [style=none] (18) at (-7.5, 15.5) {};
		\node [style=none] (19) at (-6.5, 15.5) {};
		\node [style=dot, scale=0.5] (50) at (-11.75, 18.75) {};
		\node [style=dot, scale=0.5] (51) at (-11.25, 18.75) {};
		\node [style=dot, scale=0.5] (52) at (-10.75, 18.75) {};
		\node [style=dot, scale=0.5] (53) at (-10.25, 18.75) {};
		\node [style=dot, scale=0.5] (54) at (-9.75, 18.75) {};
		\node [style=dot, scale=0.5] (55) at (-9.25, 18.75) {};
		\node [style=none] (56) at (-8.75, 19.25) {};
		\node [style=none] (57) at (-8.75, 18.25) {};
		\node [style=none] (58) at (-12.25, 19.25) {};
		\node [style=none] (59) at (-12.25, 18.25) {};
		\node [style=none] (100) at (-10.5, 22.75) {new grade 1};
		\node [style=none] (101) at (-7, 22.75) {grade 2};
		\node [style=dot, scale=0.5] (102) at (-11.75, 20.5) {};
		\node [style=dot, scale=0.5] (103) at (-11.25, 20.5) {};
		\node [style=dot, scale=0.5] (104) at (-10.75, 20.5) {};
		\node [style=dot, scale=0.5] (105) at (-10.25, 20.5) {};
		\node [style=dot, scale=0.5] (106) at (-9.75, 20.5) {};
		\node [style=dot, scale=0.5] (107) at (-9.25, 20.5) {};
		\node [style=none] (108) at (-8.75, 21) {};
		\node [style=none] (109) at (-8.75, 20) {};
		\node [style=none] (110) at (-12.25, 21) {};
		\node [style=none] (111) at (-12.25, 20) {};
		\node [style=none] (112) at (-12, 18) {};
		\node [style=none] (113) at (-9.5, 18) {};
		\node [style=none] (114) at (-10.75, 17.5) {$\vert A \vert$};
		\node [style=none] (115) at (-12, 21.25) {};
		\node [style=none] (116) at (-11, 21.25) {};
		\node [style=none] (118) at (-11.5, 21.75) {$\vert B \vert$};
		\node [style=none] (119) at (-5, 16.75) {};
		\node [style=none] (126) at (-5, 16.5) {};
		\node [style=none] (127) at (-5, 16.25) {};
		\node [style=none] (128) at (-5, 16) {};
		\node [style=dot, scale=0.5] (129) at (-11.75, 16.25) {};
		\node [style=dot, scale=0.5] (130) at (-11.25, 16.25) {};
		\node [style=dot, scale=0.5] (131) at (-10.75, 16.25) {};
		\node [style=dot, scale=0.5] (132) at (-10.25, 16.25) {};
		\node [style=dot, scale=0.5] (133) at (-9.75, 16.25) {};
		\node [style=dot, scale=0.5] (134) at (-9.25, 16.25) {};
		\node [style=none] (135) at (-8.75, 16.75) {};
		\node [style=none] (136) at (-8.75, 15.75) {};
		\node [style=none] (137) at (-12.25, 16.75) {};
		\node [style=none] (138) at (-12.25, 15.75) {};
		\node [style=dot, scale=0.5] (139) at (-11.75, 14.75) {};
		\node [style=dot, scale=0.5] (140) at (-11.25, 14.75) {};
		\node [style=dot, scale=0.5] (141) at (-10.75, 14.75) {};
		\node [style=dot, scale=0.5] (142) at (-10.25, 14.75) {};
		\node [style=dot, scale=0.5] (143) at (-9.75, 14.75) {};
		\node [style=dot, scale=0.5] (144) at (-9.25, 14.75) {};
		\node [style=none] (145) at (-8.75, 15.25) {};
		\node [style=none] (146) at (-8.75, 14.25) {};
		\node [style=none] (147) at (-12.25, 15.25) {};
		\node [style=none] (148) at (-12.25, 14.25) {};
	\end{pgfonlayer}
	\begin{pgfonlayer}{edgelayer}
		\draw (16.center) to (18.center);
		\draw (18.center) to (19.center);
		\draw (19.center) to (17.center);
		\draw (17.center) to (16.center);
		\draw (56.center) to (58.center);
		\draw (58.center) to (59.center);
		\draw (59.center) to (57.center);
		\draw (57.center) to (56.center);
		\draw (15) to (55);
		\draw (108.center) to (110.center);
		\draw (110.center) to (111.center);
		\draw (111.center) to (109.center);
		\draw (109.center) to (108.center);
		\draw (107) to (15);
		\draw (106) to (15);
		\draw (105) to (15);
		\draw (104) to (15);
		\draw (102) to (50);
		\draw (102) to (51);
		\draw (102) to (52);
		\draw (102) to (53);
		\draw (102) to (54);
		\draw (103) to (50);
		\draw (103) to (51);
		\draw (103) to (52);
		\draw (103) to (53);
		\draw (103) to (54);
		\draw [style=arrowline, in=360, out=180] (113.center) to (112.center);
		\draw [style=arrowline] (115.center) to (116.center);
		\draw (13) to (119.center);
		\draw (12) to (126.center);
		\draw (11) to (127.center);
		\draw (10) to (128.center);
		\draw (135.center) to (137.center);
		\draw (137.center) to (138.center);
		\draw (138.center) to (136.center);
		\draw (136.center) to (135.center);
		\draw (145.center) to (147.center);
		\draw (147.center) to (148.center);
		\draw (148.center) to (146.center);
		\draw (146.center) to (145.center);
		\draw (139) to (129);
		\draw (139) to (130);
		\draw (139) to (131);
		\draw (139) to (132);
		\draw (139) to (133);
		\draw (140) to (129);
		\draw (140) to (130);
		\draw (140) to (131);
		\draw (140) to (132);
		\draw (140) to (133);
		\draw (14) to (134);
		\draw (14) to (144);
		\draw (143) to (14);
		\draw (14) to (142);
		\draw (141) to (14);
	\end{pgfonlayer}
\end{tikzpicture}
    \endgroup
    \caption{Part of the construction of Theorem  \ref{thm:maxdegree} is illustrated.}
    \label{fig:2}
\end{figure}

        We require $2t-|A|-|B| \leq 0.854t$ in order to satisfy the maximum degree condition for grade 2 blocks.
        
        To ensure the new grade 1 blocks do not have too high average degree, we only need
        \begin{equation}
        \label{eq:avdeg_newgrade1}
            \frac{|A||B|}{t} + 1 < \left(\frac14 + \eps\right)t.
        \end{equation}
        The maximum degree within this pair is $\max(|A|,|B|)$, so both $|A|$ and $|B|$ can be at most $0.854t$. In the following analysis we will omit floor signs; at the end, $|A|$ and $|B|$ can be rounded up (provided $t$ is sufficiently large) without violating the conditions in the statement.

        Suppose we set $|A|=\alpha t$ for some $\alpha\in (0,1)$. Equation (\ref{eq:avdeg_newgrade1}) is satisfied as long as $|B|\leq \frac{t}{4\alpha}$ (for $t$ sufficiently large). Setting $|B|=\frac{t}{4\alpha}$, the maximum degree of grade 2 blocks is 
        \[
        2t-|A|-|B|\leq \left(2-\alpha - \frac{1}{4\alpha}\right)t.
        \]
        We take $\alpha\geq \frac12$, in which case $|A|\geq |B|$. Then
\[
\max\{|A|,|B|,2t-|A|-|B|\}= \max\left\{\alpha, 2-\alpha-\frac1{4\alpha}\right\}.
\]
       
We minimise  $\max\{\alpha,2-\alpha-\frac1{4\alpha}\}$ by setting the two functions to be equal, which yields $\alpha=\frac12+\frac1{2\sqrt{2}}<0.854$. This value now also upper bounds the maximum degree of grade 1 and grade 2 blocks. We set $n_1=2t-|A|-|B|<0.854t$, and follow the construction from Lemma \ref{lem:n} otherwise. This will set $n_2 \geq 0.25t$, and $n_j\geq n_2$ for $j\geq 2$, and so the vertices in blocks of grade at least $3$ all have maximum degree at most $0.75t$.
\end{proof}

    Loh and Sudakov \cite{LS07} introduced the notion of \emph{local degree} for a graph $G$ with a vertex partition $\PP$ as the maximum number of edges from a vertex to some block. Instead of maximum degree we can look at local degree and get a slightly better upper bound in our construction. \begin{theorem}
    \label{thm:maxblockdegree}
        For all $\eps>0$ and $t$ sufficiently large, there is a graph with a $t$-thick partition $\PP$ of $V(G)$, maximum block average degree at most $(\frac14 + \eps)t$ and local degree at most $0.731t$.
    \end{theorem}
    \begin{proof}[Proof sketch]
        We use the same construction as in the proof of Theorem \ref{thm:maxdegree}, but we use $\alpha=0.731$ instead.
        We still find $\alpha\geq \frac1{4\alpha}$, and the maximum number of edges a vertex in a block of grade 2 has to another block is hence $(1-\frac1{4\alpha})t <0.731t$. For $t$ sufficiently large, with $n_2 =(2-\alpha-\frac1{4\alpha})^{-1}\frac{t}4$, the vertices of in a block of grade 3 have maximum degree at most $t-n_2\leq 0.731t$. For vertices in blocks of higher grade, the maximum degree will be even lower due to the fact that $n_j>n_{j-1}$ for $j\geq 3$.
    \end{proof}

\section{Extension to hypergraphs}
\label{sec:hypergraphs}
In this section, we illustrate how our results translate to the setting of $r$-uniform hypergraphs. In this setting, the edges are all of size $r$ and an independent transversal must avoid containing all the vertices in an edge. Clearly, independent transversals are only obstructed by edges which are \textit{stretched} --- that is, each vertex of the edge lies in a different part of the partition. Again, exponentially many independent transversals can be found if the block average degree is sufficiently small. 
\begin{theorem}[Wanless-Wood \cite{WW20}]
\label{t:wanlesswood}
Fix integers $r \geq 2$ and $t \geq 1$. Let $G$ be an $r$-uniform hypergraph, and let $\PP=\Setx{V_1,\dots,V_n}$ be a $t$-thick partition of $V(G)$. Suppose that, for each $i \in \{1,\dots, n\}$, at most $\frac{(r-1)^{r-1}}{r^r}t^{r-1}|V_i|$ stretched edges in $G$ intersect $V_i$. Then there exist at least $\left(\frac{r-1}r\right)^nt^n$ independent transversals of $\PP$.
\end{theorem}
We show the result above is tight by extending the construction of Theorem \ref{t:main}.
\begin{theorem}
\label{t:hypergraphs}
Fix $r \geq 2$ and $\eps>0$. There is a function $C(\eps,r)$ such that for all sufficiently large $t$, and any $n > t^{C(\eps,r)}$ there exists an $r$-uniform hypergraph $G$ with a $t$-thick partition $\PP=\Setx{V_1,\dots, V_n}$ of $V(G)$ such that for each $i \in \{1,\dots, n\}$, at most $\left(\frac{(r-1)^{r-1}}{r^r}+\eps\right)t^r$ (stretched) edges in $G$ intersect $V_i$, but there is no independent transversal of $\PP$.
\end{theorem}

\begin{proof}
    We mimic the proof of Theorem \ref{t:main}. Let $\eps\in (0,1)$ and $r\geq 2$ be given, reducing $\eps$ if necessary so that $\eps < \frac{(r-1)^{r-1}}{2r^r}$. Let $\frac{2\eps}{3}<\delta<\eps$, and let $t$ be sufficiently large that
  
    \begin{equation}
    \label{eq:hypergraphstlarge}
    t^{r-1}+ \left(1+\delta\right)\frac{(r-1)^{r-1}}{r^r}t^r \leq (1+\eps) \frac{(r-1)^{r-1}}{r^r}t^r.
    \end{equation}
    Blocks of grade 1 will consist of $t$ vertices, all of which will have degree at most $t^{r-2}$, and hence the block average degree is sufficiently low by  (\ref{eq:hypergraphstlarge}).
    
    
Blocks of grade 2 consist of 
       \[
        n_2= \left\lfloor (1+\delta) \frac{(r-1)^{r-1}}{r^r} t\right \rfloor \geq \frac{\eps t}{2}
        \] 
        heavy vertices. Each such vertex $v$ has a private $(r-1)$-tuple of blocks $B_1,\dots,B_{r-1}$ and incident edges       \[
        \{\{v,b_1,\dots,b_{r-1}\}\mid b_1\in B_1,\dots,b_{r-1}\in B_{r-1}\}.
        \]
        Each block of grade 2 will be part of an $(r-1)$-tuple of blocks of grade $2$, and the remaining $t-n_2$ vertices will each be incident with at most $(t-n_2)^{r-2}$ edges. The number of edges incident to a block of grade 2 is
        \[
        t^{r-1} n_2 +(t-n_2)^{r-1} \leq (1+\eps) \frac{(r-1)^{r-1}}{r^r}t^r
        \] 
        by (\ref{eq:hypergraphstlarge}). Having defined blocks of grade $j$ for some $j\geq 2$, the blocks of grade $j+1$ are defined as follows. Once $(t-n_j)^{r-1}\leq \frac{(r-1)^{r-1}}{r^r}t^{r-1}$, we create a single block of grade $j+1$, with each vertex joined completely to a private $(r-1)$-tuple of blocks of grade $j$, and we terminate the construction. Otherwise, we set $n_{j+1}$ to be the minimum of 
        \begin{equation}
        \label{eq:njplus1hypergraph}
        \left\lfloor (1+\delta) \frac{(r-1)^{r-1}}{r^r} \left(\frac{t}{t-n_{j}}\right)^{r-1} t \right \rfloor
        \end{equation}
        and $t$.
        The edges will be defined as in the case $j = 2$, and the blocks will form $(r-1)$-tuples.
        
        

        By (\ref{eq:hypergraphstlarge}), the number of edges incident to a block of grade $j+1$ is
        \[
        n_{j+1}(t-n_j)^{r-1}+(t-n_{j+1})^{r-2}\leq (1+\eps) \frac{(r-1)^{r-1}}{r^r}t^r.
        \]
       Again, it suffices to show that the construction terminates, for which it suffices to show that $n_j<n_{j+1}$ when $n_j,n_{j+1}<t$. For $t$ sufficiently large,
        \begin{align*}
           \frac{n_{j+1}}{t} &= \left\lfloor (1+\delta) \frac{(r-1)^{r-1}}{r^r} \left(\frac{t}{t-n_{j}}\right)^{r-1} t \right \rfloor t^{-1}\\
           &\geq (1+\frac{\delta}{2}) \frac{(r-1)^{r-1}}{r^r} \left(\frac{1}{1-n_{j}/t}\right)^{r-1} \\
           &\geq (1+\frac{\delta}{2})\frac{n_{j}}{t},
        \end{align*}
        using that $(1-x)^{r-1}x\leq \frac{(r-1)^{r-1}}{r^r}$ for all $x\in (0,1)$.
        
        Observe that $\frac{n_2}{t} > \frac{\eps}2$, and for $i\geq 3, \frac{n_i}{n_{i-1}} \geq (1+\frac{\eps}3)$ (or $n_i = t$). As such, for some $j\leq \log_{1+\eps/3}(\frac2{\eps} + 2)$, $n_j = t$, and the construction terminates at that grade.
        
       Each block of grade $j$ is joined to $(r-1)n_j$ (private) blocks of grade $j-1$. Each of these blocks in turn has private blocks of grade $j-2$. Overall, if $j$ is such that $n_j = t$, we require at most $((r-1)t)^{j+1}$ blocks. We can obtain constructions for any larger $n$ by adding additional blocks with no incident edges. 
        
\end{proof}
We also offer the following analogue of Theorem \ref{thm:maxdegree}.

\begin{theorem}
In the construction of Theorem \ref{t:hypergraphs}, we can also require \begin{align*}\Delta(G) \leq \left(1-\frac{(r-1)^{r-1}}{3r^r}\right)t^{r-1}\end{align*}
\end{theorem}

For brevity, we only provide an outline of how the above construction is modified, and do not provide a detailed proof. Let $c_r = \frac{(r-1)^{r-1}}{r^r}$.

We bounded the maximum degree for the forest construction by replacing each block of grade 1 with two blocks, and placed an unbalanced complete bipartite graph between them. The natural generalisation of this to hypergraphs is to replace each grade 1 block with $r$ blocks, and place a complete $r$-partite $r$-uniform hypergraph between them, with part sizes $m_1\geq \dotsc\geq m_r$. We cannot choose all vertices from this graph, and so the remaining vertices form a forced set, of size $rt - \sum m_i$. 

As before, the maximum vertex degree must occur in either a grade 1 or a grade 2 block --- beyond this point, since the $n_i$ are increasing, the degrees of the heavy vertices are decreasing. We will ignore all edge contributions of order less than $t^{r-1}$: this is admissible, since we can if necessary make $\eps$ smaller, and take $t$ sufficiently large. The degrees in the grade 1 block are at most $m_1m_2\dotsc m_{r-1}$ (arising from vertices in a part of minimal size), and the degrees in the grade 2 block are at most $(rt-\sum m_i)^{r-1}$. Our problem is thus to minimise the maximum of these two values, given the constraint that the product of the $m_i$ is at most $c_rt^r$.

These parameters can be optimally picked by reduction to a three variable optimisation problem. Taking two (non-minimal) $m_i$ and slightly increasing the larger, while decreasing the smaller will improve our objective function. We can therefore assume that the only values which appear are $m_r$, $t$ and possibly one block with an intermediate size.

If we have multiple $m_i$ which are equal and minimal, we cannot get a large average degree from the grade 1 blocks, and so we can increase all $m_i$ without violating a constraint. Therefore, we can assume there is only one smallest block. We now have an optimisation problem in two variables $m_r$ and $m_{r-1}$, and reduce this to one variable by letting  $m_rm_{r-1} = c_rt^{r-1}$. One can check that substituting $m_{r-1} = (1-\frac{c_r}{3})t$ gives the claimed upper bound; this gives a good approximation to the general exact solution.

\section{Conclusion}
\label{sec:concl}
Let us conclude with some remarks, and open questions for further investigation.

\subsection{Local degree} Loh and Sudakov~\cite{LS07} proved that for every $\eps>0$ there exists
$\gamma>0$ such that if $\PP$ is a $t$-thick partition of $V(G)$, the
maximum degree of $G$ is at most $(1-\eps) t$ and the local degree of $G$
is at most $\gamma t$, then $\PP$ has an independent transversal. This
was extended by Kang and Kelly~\cite{KK20} and independently by Glock
and Sudakov~\cite{GS21} to the case where the maximum degree is
replaced by the maximum block average degree.

As noted by Kang and Kelly~\cite{KK20}, this could possibly be further
strengthened by introducing a parameter they called `maximum average
colour multiplicity'. Let $G$ be a graph with a partition $\PP$ of
$V(G)$. We define $\overline\mu(G)$ as the maximum, taken over pairs
$A,B$ of blocks of $\PP$, of the average number of neighbours in $B$
of a vertex $v\in A$. Kang and Kelly propose a problem which we
formulate as follows.

\begin{problem}[\cite{KK20}]
  Is it true that for every $\eps>0$ there exists $\gamma>0$ such that
  if $t$ is sufficiently large, $\PP$ is a $t$-thick partition of
  $V(G)$, the maximum block average degree of $G$ is at most $(1-\eps) t$ and
  $\overline\mu(G) \leq \gamma t$, then $\PP$ has an independent
  transversal?
\end{problem}

An affirmative answer would imply a conjecture of Loh and
Sudakov~\cite{LS07} where the maximum block average degree is replaced by
the maximum degree and the condition $\overline\mu(G) \leq \gamma t$
is replaced by the number of edges between any two blocks of $\PP$
being at most $\gamma t^2$.

\subsection{Maximum degree and maximum block average degree.} 
It is natural to ask if the tightness example provided by Theorem~\ref{t:main} can
be combined with the tightness examples for Theorem~\ref{t:haxell}
discussed in Section~\ref{sec:introduction}.

\begin{problem}
  Is it true that for every $\eps > 0$ and every sufficiently large
  $t$, there is a graph $G$ with a $t$-thick partition $\PP$ of $V(G)$
  such that $\Delta(G) \leq (\frac12+\eps)t$, the maximum block
  average degree of $G$ is at most $(\frac14+\eps)t$ and $\PP$ admits
  no independent transversal in $G$?
\end{problem}

Theorem~\ref{thm:maxdegree} provides a step in this direction, but the
maximum degree bound of $0.854t$ is a long way from
$(\frac12+\eps)t$.

More generally, what can we say about the set of pairs $(\alpha,\beta)
\in [0,1]^2$ such that for sufficiently large $t$, there is a graph
with a $t$-thick partition $\PP$, maximum degree at most $\alpha t$,
maximum block average degree at most $\beta t$ and no independent
transversal of $\PP$?

For $r$-uniform hypergraphs, we can ensure a `tight example' while each vertex meets at most $\eta t^{r-1}$ edges, for $\eta(r) \sim (1-\frac{1}{3er})$. Although we can ensure the number of edges meeting each part is $o_r(1)t^r$, we leave open whether this remains true under a `maximum degree' condition of $o_r(1)t^{r-1}$, or even $0.99t^{r-1}$.

\begin{problem}
  Is there a constant $0<\eta<1$ such that for any $r,\epsilon$ there is an $r$-uniform hypergraph $H$ and vertex partition $\PP = (V_i)$ with at most $\big( \frac{(r-1)^{r-1}}{r^r} + \epsilon\big)t^r$ stretched edges meeting each $V_i$, each vertex meeting at most $\eta t^{r-1}$ edges, and no independent transversal of $\PP$? Can we do the same for some $\eta(r)\rightarrow 0$ ?
\end{problem}

\subsection{Avoiding complete bipartite graphs.} The examples with a bound on the maximum degree
given by Theorem~\ref{thm:maxdegree} contain complete bipartite
subgraphs of order linear in the block size. We may ask if similar
examples may be constructed if we exclude this, say by forbidding
cycles of length $4$ in the graph.

\begin{problem}
  Is there a $\delta > 0$ such that for every $\eps > 0$ and all
  sufficiently large $t$, there is a graph $G$ with a $t$-thick
  partition $\PP$ of $V(G)$ such that $G$ contains no $4$-cycle, the
  maximum degree of $G$ is at most $(1-\delta)t$, the maximum block average
  degree of $G$ is at most $(\frac14 + \eps)t$, and $\PP$ has no
  independent transversal in $G$?
\end{problem}

\subsection{Number of blocks.}
The number of blocks of the vertex partitions of the graphs provided by our proof of Theorem~\ref{t:main} is large with respect to the block size. If we fix the number of blocks and let the block size vary, how will it affect the existence of independent transversals in relation to maximum block average degree?

We remark that the corresponding question for the maximum degree variant of the problem is well understood. Let $\Delta(n,t)$ be the largest integer such that for every graph $G$ with maximum degree $\Delta(n,t)$, every $t$-thick partition of $V(G)$ with $n$ blocks has an independent transversal. It follows from the results of~\cite{A03,HS06,ST06} that
\begin{equation*}
    \Delta(n,t) = \left\lceil\frac{nt}{2(n-1)}\right\rceil.
\end{equation*}

\subsection{Unequal part sizes}
We finish with a remark. One might be tempted to modify the average degree
assumption in Theorem~\ref{t:ww-graphs} as follows: If $G$ is a graph
and $\PP$ a partition of $V(G)$ such that each block $B$ is
incident with at most $\size{B}^2/4$ edges, then there is an
independent transversal of $\PP$ in $G$. This fails, however --- even
if the constant $1/4$ is replaced with an arbitrary $\beta > 0$. To
see this, take a large integer $k$ and let $G$ be the disjoint union
of stars $S_1,\dots,S_{k^2}$ of $k+1$ vertices each. Let the set $V_0$
consist of the centers of the stars and let $V_i$ ($1\leq i\leq k^2$)
consist of the leaves of $S_i$. Then the partition
$\Setx{V_0,\dots,V_{k^2}}$ has no independent transversal in $G$,
since for any choice of a vertex $v\in V_0$ there is a block
consisting of neighbours of $v$. At the same time, any block $B$ is
incident with at most $\size{B}^2/k$ edges.

\subsection*{Acknowledgements}
This work was started at the Entropy Compression and Related Methods workshop in March 2021. We extend our thanks to all participants and organisers, and especially to Ross Kang for useful pointers to the literature and many interesting discussions. The first author would like to thank Han Peters for explaining the background related to M\"obius transformations. We would like to thank the anonymous reviewers for their helpful comments on this manuscript, and in particular for bringing Remark~\ref{R:reviewercomment} to our attention.

\bibliographystyle{plain}
\bibliography{refs.bib}

\end{document}